\documentclass[a4paper,12pt]{amsart}
\usepackage[letterpaper, margin=1.2in]{geometry}
\usepackage{latexsym}
\usepackage{amssymb}
\usepackage{amsmath}
\usepackage{amsfonts}
\usepackage[table]{xcolor}
\usepackage{pgfplots}
\usepackage{color}

\usepackage{txfonts,pxfonts} 
\usepackage{tikz}
\usetikzlibrary{calc,arrows}
\usetikzlibrary{calc}
\usepackage{verbatim}
\newtheorem{thm}{Theorem}[section]

\newtheorem{lem}[thm]{Lemma}

\theoremstyle{definition}

\newtheorem{rem}{Remark}

\def\fph{\mathbb{F}_{\ph}}

\newcommand{\Z}{\mathbb Z}

\newcommand{\Q}{\mathbb Q}

\newcommand{\F}{\mathbb F}

\def\F{\mathbb{F}}
\DeclareMathAlphabet{\mathpzc}{OT1}{pzc}{m}{it}

\newcommand{\p}{\mathfrak{p}}
\def\ol{\overline}
\def\al{\alpha}
\def\la{\lambda}

\def\om{\omega}
\def\th{\theta}
\def\md#1{\ \mbox{\rm(mod }{#1})}
\def\nph#1{N_{\ph}(#1)}
\def\npp#1{N_{\ph}^+(#1)}
\def\ph{\phi}
\newcounter{cs}
\stepcounter{cs}
\newcommand{\casos}{\begin{itemize}}
\newcommand{\fcasos}{\end{itemize}\setcounter{cs}{1}}

\newfont{\tit}{cmr12 scaled \magstep3}
\begin{document}
\title[]{On the index divisors of quintic number fields defined by $x^5+ax+b$}
\textcolor[rgb]{1.00,0.00,0.00}{}
\author{Lhoussain El Fadil}
\maketitle
\centerline{Faculty of Sciences Dhar El Mahraz, P.O. Box  1796 Atlas-Fez,}
\centerline{Sidi Mohamed Ben Abdellah University, Fez-- Morocco}
\centerline{ lhoussain.elfadil@usmba.ac.ma}
\begin{abstract}
     The  goal of  this paper is to calculate explicitly the field index of  any quintic number  field $K$ generated  by a complex  root  $\al$ of a monic irreducible trinomial  $F(x) = x^5+ax+b \in \Z[x]$. In such a way we provide  a complete answer to the  Problem 22 of  Narkiewicz \cite{WN}. Namely for every prime integer $p$, we  evaluate the highest power of $p$ dividing $i(K)$. In particular, we give sufficient conditions on $a$ and $b$, which guarantee the  non monogenity of $K$. Finally, we illustrate our results by some computational examples.        
\end{abstract}
\section{Introduction} 
Let $K$ be an number field generated by a root $\al$ of an irreducible polynomial $F(x)\in \Z[x]$. Let $\Z_K$  be its ring of integers. It is well known that $\Z_K$ is a free abelian group of rank $n$; the degree of $F(x)$. It follows that for any primitive element $\al\in \Z_K$ of $K$, the quotient group $\Z_K/\Z[\al]$ is a finite abelian  group is finite. For such a primitive element $\al\in \Z_K$ of $K$, let $ind(\al) = (\Z_K: \Z[\al])$ be the  cardinal order of  $\Z_K/\Z[\al]$, called the index of $\Z[\al]$ in $\Z_K$. The greatest common divisor of the indices of all integral primitive elements of $K$ is called the  index of $K$, and  denoted by $i(K)$. Say $i(K)=\gcd \ \{ ( \Z_K; \Z [\alpha]) \, |\,\alpha \in {\Z}_K \mbox{ and } K=\Q(\al) \}$. A rational prime $p$ dividing $i(K)$ is called a prime common  divisor of $K$. It is clear that if $\mathbb{Z}_K$ has a power integral basis, then the index of $K$ is trivial, namely $i(K)=1$. Therefore a field having a prime divisor is not monogenic. 
 The first  number field with non trivial index was given by Dedekind in 1871,
 	 who exhibited examples in cubic and quartic number fields. For example, he considered the cubic field $K$ generated by a complex root of $x^3-x^2-2x-8$ and  showed that the prime $2$ splits completely in $K$. So, if we suppose that $K$ is monogenic, then we would be able to find a cubic polynomial generating $K$, that splits completely into distinct polynomials of degree $1$ in $\mathbb{F}_2[x]$. Since there are only $2$ distinct polynomials of degree $1$ in $\mathbb{F}_2[x]$, this is impossible. 
In \cite{En2}, for any number field of degree $n\le 7$, Engstrom gave an explicit formula which relates   $\nu_p(i(K))$ to the factorization of  $(p)$ into powers of prime ideals of $K$ for every positive prime $p\le n$. Problem 22 of Narkiewicz \cite{WN} asks for an explicit formula for the highest power of a
given prime integer $p$ dividing $i(K)$, say $\nu_p(i(K))$.
In \cite{Na}, Nakahara  studied the  index of non-cyclic but abelian biquadratic number fields. 
 In \cite{GPP}  Ga\'al et al. characterized  the field indices of biquadratic number fields having Galois group $V_4$. In \cite{DA} for any  quartic number field $K$ defined by a trinomial $x^4 + ax + b$, Davis and  Spearman gave necessary and sufficient conditions on $a$ and $b$, which characterize when $2$ divides $i(K)$.  
Also in   \cite{FG} for any  quartic number field $K$ defined by a trinomial $x^4 + ax^2 + b$, El Fadil and Ga\'al gave necessary and sufficient conditions on $a$ and $b$, which characterize when a prime integer $p$ divides $i(K)$.
  In \cite{E5} for any  quintic number field $K$ defined by a trinomial $x^5 + ax^2 + b$, we gave necessary and sufficient conditions on $a$ and $b$, which characterize when a prime integer $p$ divides $i(K)$. In \cite{ENT, E6} for any  sextic number field $K$ defined by a trinomial, we gave necessary and sufficient conditions, which characterize when a prime integer $p$ divides $i(K)$. 
  In this paper, for any quintic number fields $K$ defined by a trinomial $x^5 + ax + b\in \Z[x]$, we calculate explicitly  the field index of $K$. 
  Contrary to the quoted list, except our quartic paper with Ga\'al \cite{FG4},  we haven't even thought of introducing the value $i(K)$ of the index of the field $K$. Moreover it is the first time that the index $i(K)$ is completely calculated.
Based on Engstrom's results given in \cite{En2},  the unique prime candidates to divide $i(K)$ are $2$ and $3$, say $i(K)=2^{\nu_2}3^{\nu_3}$ , with $0\le \nu_2\le 5$ and $0\le \nu_3\le 2$.  In Theorems \ref{valueind} shows that $0\le \nu_2\le 2$ and $ \nu_3=0$, say $i(K)\in\{1,2,4\}$.
  \section{main results}
Let $a$ and $b$ be two rational integers such that $F(x)=x^5+ax+b$ is irreducible over $\Q$, and let $K = \Q(\al)$ be a number field generated by a root $\al$ of  $F(x)$.
 If $\nu_p(a)\ge 4$ and $\nu_p(b)\ge 5$ for a prime integer $p$, then $\th=\frac{\al}{p}$ is a primitive integral element of $K$, which is a root of the irreducible 
trinomial $x^5 + Ax +B$  having integer coefficients $A=a/p^4$ and $B=b/p^5$. So up to replace $\al$ by $\theta$ and  repeat this process until to get  
either $\nu_p(a)\le 3$ or $\nu_p(b)\le 4$, we can assume that $\nu_p(a)\le 3$ or $\nu_p(b)\le 4$ for every prime integer $p$.
Let $\Z_K$ be the ring of integers of $K$, $\triangle(F)=(2^8 a^5+5^5 b^4)$ the discriminant of $F(x)$ and $d_K$ the absolute discriminant of $K$. Then we have the following well know formula: 
 \begin{eqnarray}\label{indexdisc}
 \triangle(F)=\pm ind(\al)^2d_K,
\end{eqnarray} which relates the index and $\triangle(F)$  and $d_K$.

The following theorem characterizes, when is $\Z[\al]$ integrally closed? It improves \cite[Theorem 3.1]{Smt}, which gave sufficient conditions on $a$ and which guarantee $\Z[\al]$ integrally closed.
\begin{thm}\label{mono}
		Let $K$ be a quintic number field generated by a  root $\al$ of an irreducible trinomial $F(x) = x^{5} +a  x + b  \in \Z[x] $  with $\nu_p(a) \le 3$ or  $\nu_p(b) \le 4 $ for every prime integer $p$.  Then $\Z[\al]$ is  integrally closed if and only if the following hold:
	\begin{enumerate}
			\item
		 For every prime integer $p$, if $p$ divides $a$ and $b$, then  $p^2$ does not divide $b$.
		\item
		$2$ does not divide $b$ or  $b\equiv 1-a\md4$.
			\item
		If $5$ divides $a$, then $5$ does not divide $b$ and $\nu_5(-1+a+b^4)\le 1$.
	\item
	For every prime integer $p\not\in\{2,5\}$, if $p$ does not divide both $a$ and $b$, then $\nu_p(2^8 a^5 + 5^5 b^4)\le 1$.		
		\end{enumerate}
In particular, if these conditions hold, then  $K$ is monogenic and 
 $i(K)=1$.
\end{thm}
\begin{rem}
    The number field $K$ can be monogenic although $\Z[\al]$ is not integrally closed.    
\end{rem}
Indeed, let $K$ be the quintic number field generated by a complex root of $F(x)=x^5+8x+4$. Let $\phi=x$. Since $\ol{F(x)}=\ph^5(x)$ in $\F_2[x]$ and $\nph(F)=S$ has a single side of degree $1$, we conclude that $F(x)$ is irreducible over $\Q$. Since $\triangle(F)=2^8\cdot 11\cdot 13\cdot 251$, the unique prime which is candidate to divide $ind(\al)$ is $2$. Let $\th=\frac{\al^3}{2}$, then $G(x)=x^5+32x-12x^2+2$ is the minimal polynomial of $\th$ over $\Q$. Thus $K=\Q(\th)$. By using Dedekind's criterion, we conclude that $\nu_2(ind(\th))=0$. Let $p$ be a positive odd prime. Since $\Z[\al]\subset \Z[\th]\subset \Z_K$,  we conclude that if $p$ divides $(ind(\th))$, then $p$ divides $(ind(\al))$. Therefore we conclude that $\nu_p(ind(\th))=0$ for every prime integer $p$. Which means that $\Z_K=\Z[\th]$.\\

   Our second main result is  the following theorem, which completely provides the value $i(K)$. Let $\triangle=2^8 a^5 + 5^5 b^4$ and $\triangle_2=\frac{\triangle}{2^{\nu_2(\triangle)}}$. 
\begin{thm}\label{valueind}
	$$\begin{array}{|c|c|c|}
	\hline
 {\mbox conditions}& i(K)\\
 \hline
	     a\equiv 4\md{64} { \mbox{ and }}  \nu_2(b)\ge 6&2\\
	    \hline
      a\equiv 12\md{32}{ \mbox{ and }}  \nu_2(b)\ge 5&2\\
      \hline
	    a\equiv 20\md{64} { \mbox{ and }}  \nu_2(b)=5&2\\
	     \hline
	     a\equiv 1\md{4}{ \mbox{ and }} \nu_2(b+1+a)\ge 2&2\\
	      \hline
	     a\equiv 3\md{8}, \nu_2(b+1+a)=3, \nu_2(\triangle) { \mbox{ is odd }} { \mbox{ and }} \triangle_2\equiv -a\md4&4\\
    \hline
    a\equiv 3\md{8},  \nu_2(b+1+a)=3{ \mbox{ and }} \nu_2(\triangle) { \mbox{ is even}}&2\\
    \hline
	     a\equiv 3\md{8} { \mbox{ and }} \nu_2(b+1+a)=4& 2\\
	     \hline
	     a\equiv 3\md{16} { \mbox{ and }} \nu_2(b+1+a)\ge 6& 4\\
	     \hline
	     a\equiv 11\md{16}, \nu_2(b+1+a)\ge 5, \nu_2(\triangle) { \mbox{ is odd }} { \mbox{ and }} \triangle_2\equiv -a\md4&4\\
	     \hline
       a\equiv 11\md{16}, \nu_2(b+1+a)\ge 5 { \mbox{ and }} \nu_2(\triangle) { \mbox{  is even}}&2\\
	     \hline
	     {\mbox Otherwise }&1\\
 
   \hline
	     	\end{array}$$
       In particular, if $i(K)\neq 1$, then $K$ is not monogenic.
\end{thm}
\section{A short introduction to prime ideal factorization based on Newton polygons}
\label{intro}

In 1894, Hensel
developed a powerful approach by showing that for every prime integer $p$, the prime ideals of $\Z_K$
lying above a  $p$ are in one--one correspondence with
monic irreducible factors of $F(x)$ in $\Q_p[x]$. For every prime ideal
corresponding to any irreducible factor in $\Q_p[x]$, the
ramification index  and the residue degree together are the same as
those of the local field defined  by the associated irreducible factor
\cite{H}. Since then, to factorize $p\Z_K$, we need to factorize $F(x)$ in $\Q_p[x]$.  Newton's polygon techniques can be used to refine the factorization. This is a standard method which is rather technical but
very efficient to apply. We have introduced the corresponding concepts in several former papers.
Here we only give a brief introduction which makes our proofs understandable.
For a detailed description we refer to Section $3$ of El Fadil and Ga\'al's  paper \cite{FG}, and to Guardia, Montes and Nart's paper \cite{GMN}.	
	For every prime integer $p$, let $\nu_p$ be the $p$-adic valuation of $\Q_p$ and $\Z_p$ the ring of $p$-adic integers.  Let $F(x)\in\mathbb{Z}_p[x]$ be a monic polynomial and $\phi\in\mathbb{Z}_p[x]$  a monic lift of an irreducible factor of $\ol{F(x)}$ modulo $p$. Let $F(x)=a_0(x)+a_1(x)\phi(x)+\cdots+a_n(x)\phi(x)^l$ be the $\phi$-expansion of $F(x)$, $\nph{F}$ the $\phi$-Newton polygon of $F(x)$ and $\npp{F}$ its  principal part. Let $\mathbb{F}_{\phi}$ be the field $\mathbb{F}_p[x]/(\overline{\phi})$. For every side $S$ of $\npp{F}$ with length $l$ and initial point $(s,u_s)$,  for every $i=0,\ldots,l$, let  $c_i\in\mathbb{F}_{\phi}$ be the residue coefficient, defined as follows: 
	$$c_{i}=
	\left
	\{\begin{array}{ll} 0,& \mbox{ if } (s+i,{\it u_{s+i}}) \mbox{ lies strictly
			above } S,\\
		\left(\dfrac{a_{s+i}(x)}{p^{{\it u_{s+i}}}}\right)
		\,\,
		\mod{(p,\phi(x))},&\mbox{ if }(s+i,{\it u_{s+i}}) \mbox{ lies on }S.
	\end{array}
	\right.$$
	Let $-\lambda=-h/e$ be the slope of $S$, where $h$ and $e$ are two positive coprime integers. Then  $d=l/e$ is the degree of $S$. Let ${R_{\lambda}(F)(y)}=t_dy^d+t_{d-1}y^{d-1}+\cdots+t_{1}y+t_{0}\in\mathbb{F}_{\phi}[y]$, called  
	the residual polynomial of $F(x)$ associated to the side $S$, where for every $i=0,\dots,d$,  $t_i=c_{ie}$. If ${R_{\lambda}(F)(y)}$ is square free for each side of the polygon $\npp{F}$, then we say that $F(x)$ is $\phi$-regular.\\ 
	Let $\overline{F(x)}=\prod_{i=1}^{r}\ol{\phi_i}^{l_i}$ be the factorization of $F(x)$ into powers of monic irreducible coprime polynomials over $\mathbb{F}_p$, we say that the polynomial $F(x)$ is $p$-{regular} if $F(x)$ is a $\phi_i$-regular polynomial with respect to $p$ for every $i=1,\dots,r$. Let  $N_{\phi_i}^+(F)=S_{i1}+\cdots+S_{ir_i}$ be the $\phi_i$-principal Newton polygon of $F(x)$ with respect to $p$. For every $j=1,\dots,r_i$, let $R_{\lambda_{ij}}(y)=\prod_{s=1}^{s_{ij}}\psi_{ijs}^{a_{ijs}}(y)$ be the factorization of $R_{\lambda_{ij}}(y)$ in $\mathbb{F}_{\phi_i}[y]$. Then we have the following  theorem of index of Ore:
	\begin{thm}\label{ore}$($\cite[Theorem 1.7 and Theorem 1.9]{EMN}$)$\\
		Under the above hypothesis, we have the following:
		\begin{enumerate}
			\item 
			$$\nu_p((\mathbb{Z}_K:\mathbb{Z}[\alpha]))\geq\sum_{i=1}^{r}\text{ind}_{\phi_i}(F).$$  
			The equality holds if $F(x) \text{ is }p$-regular.
			\item 
			If $F(x) \text{ is }p$-regular, then
			$$p\mathbb{Z}_K=\prod_{i=1}^r\prod_{j=1}^{t_i}\prod_{s=1}^{s_{ij}}\mathfrak{p}_{ijs}^{e_{ij}}$$
			is the factorization of $p\mathbb{Z}_K$ into powers of prime ideals of $\mathbb{Z}_K$, where $e_{ij}$ is the smallest positive integer satisfying $e_{ij}\la_{ij}\in \Z$ and the residue degree of $\mathfrak{p}_{ijs}$ over $p$ is given by $f_{ijs}=\deg(\phi_i)\times \deg(\psi_{ijs})$ for every $(i,j,s)$.
		\end{enumerate}
	\end{thm}
	\smallskip
    When  the theorem  of Ore fails, that is $F(x)$ is not $p$-regular, then in order to complete the factorization of $F(x)$, Guardia, Montes, and Nart introduced the notion of {\it high order Newton polygon}.   For more details, we refer to \cite{GMN}.
    \section{Proofs of  main results}
		\begin{proof}[Proof of Theorem \ref{mono}]	
	Since $\triangle(F)=5^5\cdot b^4+2^8\cdot a^5 $, thanks to the formula (\ref{indexdisc}), it suffices to show that   $\nu_p(ind(\al))=0$ for every prime integer dividing 	$(2^8\cdot a^5 + 5^5\cdot b^4)$.
		\begin{enumerate}
		\item
		For every prime integer  $p$, if $p$  divides $a$ and $b$, then by By Theorem \ref{ore}, $p^2$ does not divide $b$ is a necessary and sufficient condition so that $p$ does not divide $ind(\al)$.
		    \item 
		    For $p=2$, if $2$ does not divide $b$, then $2$ does not divide $\triangle(F)$, and so $\nu_2(ind(\al))=0$. If  $2$  divide $b$ and does not divides $a$, then $\ol{F(x)}=x\ph(x)^4$, where $\ph(x)= x-1$.  By Theorem \ref{ore},  $\nu_2(ind(\al))=0$ if and only if $ind_\ph(F)=0$, which means that $\nu_2(1+a+b)=1$.  
		    \item 
		    For $p=5$, if $5$ does divide $a$, then $5$ does not divide $\triangle(F)$, and so $5$ does not divide $ind(\al)$. If $5$    divides $a$ and does not divides $b$, then $\ol{F(x)}=\ph(x)^5$, where $\ph(x)= x+b$.  By  Theorem \ref{ore},  $\nu_p(ind(\al))=0$ if and only if $ind_\ph(F)=0$, which means that
		    $\nu_5(b-ab-b^5)=1$. That means  $\nu_5(1-a-b^4)=1$.   
		    \item 
		    For $p\not\in\{2,5\}$, if $p$ divides $a$ and $p$ does divide $b$ or $p$ divides $b$ and $p$ does divide $a$, then $p$ does not divide  $\triangle(F)$. By the formula (\ref{indexdisc}),  $p$ does not divide $ind(\al)$. For the same reason, if    $p^2$ does not divide  $\triangle(F)$, then  $p$ does not divide  $ind(\al)$. Now assume that  $p$ does not divide $ab$ and $p^2$ divides  $\triangle(F)$. A square root of $\ol{F(x)}$ if there exists is an element $\ol{u}\in \F_p$, defined by $u\in\Z$ such that $\ol{F(u)}=\ol{F'(u)}=0$. That is  $u\in\Z_p$ with $4au\equiv -5b \md{p^r}$ for an adequate natural integer $r$. Let $u=\frac{-5b}{4a}$. Since $p$ does not divide $4a$, then $u\in\Z_p$. Let $F(x+u)=x^5+5ux^4+10u^2x^3+10u^3x^3+Ax+B$  with $A=5u^4+a$ and $B=b+au+u^5$. So $A=\frac{ 2^8a^5+5^5b^4}{(4a)^4}$ and $B=\frac{-b(2^{8}a^5+5^5b^4}{2^{10}a^5}= -b\frac{ \triangle(F)}{2^{10}a^5}$. Thus $\nu_p(A)=\nu_p(B)=\nu_p(\triangle(F))$. So for $\ph=x-u$, $\ol{ph(x)}^2$ divides  $\ol{F(x)}$ and  $\npp{F}=S$ has a single side joining $(0,\nu_p(\triangle(F))$ and $(2,0)$.  It follows by Theorem \ref{index},  that $ind(\al)\ge ind_\ph(F)\ge \lfloor \triangle(F)/2 \rfloor\ge 1$, where for every real number $x$, $\lfloor x \rfloor$ is the integral part of $x$.    
	\end{enumerate}
\end{proof}
		 
			{For the proof of Theorem \ref{valueind}, we need the following lemma, which  characterizes the prime  divisors of $i(K)$}.
			
	 	\begin{lem}\label{index}
		Let $p$ be a rational prime integer and $K$ be a number field. For every positive integer $f$, let $\mathcal{P}_f$ be the number of distinct prime ideals of $\mathbb{Z}_K$ lying above $p$ with residue degree $f$ and $\mathcal{N}_f$ the number of monic irreducible polynomials of $\mathbb{F}_p[x]$ of degree $f$.  Then $p$ divides $i(K)$ if and only if $\mathcal{P}_f>\mathcal{N}_f$ for some positive integer $f$.
	\end{lem}
	\begin{proof}[Proof of Theorem \ref{valueind}]
First, we will show that for every odd prime integer, $\nu_p(i(K))=0$. By virtue of Lemma \ref{index} and Engstrom's results, we know that for every prime integer $p\ge 5$, $p$ does not divide $i(K)$. So, we have to show that $3$ does not divide $i(K)$.  By virtue of Engstrom's results, we need to factorize the ideal $(3)$ into powers of prime ideals of $\Z_K$ lying above $3$. Recall  that if $\ol{F(x)}$ is square free in $\F_3[x]$, then by Dedekind's criterion, $3$ does not divide $(\Z_K:\Z[\al])$, and so  $3$ does not divide $i(K)$. Assume that  $\ol{F(x)}$ has a square factor in $\F_3[x]$. That is $(a,b)\in\{(0,0),(1,1),(1,-1)\}\md3$. In these cases $3$ divides $i(K)$ if and only if there are at least four prime ideals of $\Z_K$ lying above $3$ with residue degree $1$ each.
	\begin{enumerate}
	\item 
	If $(a,b)\equiv (0,0)\md3$, then $\ol{F(x)}= \ph^5$ in $\F_3[x]$ with $\ph = x$. If $\nph{F}=S$ has a single side, then since by assumption, $\nu_3(a)\le 1$ or $\nu_3(b)\le 4$, we conclude that  $d(S)=1$, and so by Theorem \ref{ore}, $3\Z_K=\p_{111}^5$, with $\p_{111}$ the unique prime ideal of $\Z_K$ lying above $3$. 
	If $\nph{F}=S_1+S_2$ has two sides, then by assumption, $\nu_3(a)\le 3$ and $\nu_3(a) < \nu_3(b)$. Thus 
	$d(S_1)=1$ and $d(S_2)\in\{1,2\}$. It follows that $S_1$ provides a unique prime ideal of $\Z_K$ lying above $3$ with residue degree $1$ and ramification index $1$ and $S_2$  provides at most two prime ideals of $\Z_K$ lying above $3$ with residue degree $1$ each. Therefore $3$ does not divide $i(K)$.
		\item 
		If $(a,b)\equiv (1,1),(1,-1)\}\md3$, then $\ol{F(x)}= \ph^2(x^3-x^2+1)$ in $\F_3[x]$  with $\ph = x-1$.  By Theorem \ref{ore}, there are at most two prime ideals of $\Z_K$ lying above $3$ with residue degree $1$ each. Thus  $3$ does not divide $i(K)$.
			\item 
		Similarly if $(a,b)\equiv (1,-1)\}\md3$, then $\ol{F(x)}= \ph^2(x^3+x^2-1)$ in $\F_3[x]$  with $\ph = x+1$.  By Theorem \ref{ore}, there are at most two prime ideals of $\Z_K$ lying above $3$ with residue degree $1$ each. Thus  $3$ does not divide $i(K)$.
		\end{enumerate}
  Now let us calculate   $\nu_2(i(K))$.
By virtue of Lemma \ref{index} and Engstrom's results, we need to factorize the ideal $(2)$ into powers of prime ideals of $\Z_K$ lying above $2$. Since $\triangle(F)=5^5 b^4+2^8 a^5$, by the formula (\ref{indexdisc}), we conclude that if $2$ does not divide $b$, then  $2$ does not divide $ind(\al)$, and so does not divide  $i(K)$. Now assume that $2$ divides $b$. 
	\begin{enumerate}
	\item 
	If	$2$  divides both of  $a$ and $b$, then $ \ol{F(x)}= \ph^5$ in $\F_2[x]$ with $\ph = x$. It follows that 
	if $\nph{F}=S$ has a single side, then since by assumption, $\nu_2(a)\le 3$ or $\nu_2(b)\le 4$, we conclude that  $d(S)=1$, and so by Theorem \ref{ore},  there is a unique prime ideal of $\Z_K$ lying above $2$. If  $\nph{F}=S_1+S_2$ has two sides; $5\nu_2(a)-1\ge 4\nu_2(b)$, then  as by assumption, $\nu_2(a)\le 3$ or  $\nu_2(b)\le 4$, we conclude that $\nu_2(a)\in\{ 1,2,3\}$ and $\nu_2(a) < \nu_2(b)$. Thus  $d(S_1)=1$ and $d(S_2)\in\{1,2\}$. If $d(S_2)=1$, then $S_2$ provides a unique prime ideal of $\Z_K$ lying above $2$ with residue degree $1$ and ramification index $1$. Therefore if $\nu_2(a)\in\{1,3\}$, then  by Theorem \ref{ore}, $2\Z_K=\p_{111}^4\p_{121}$. If $\nu_2(a)=2$,  then $S_2$ is of degree $2$ and its attached residual polynomial  is $F_{S_1}(y)=(y+1)^2$. Thus we have to use second order Newton polygon techniques. For this reason let $\ph_2=x^2+2$, $\om$ the second order valuation associated to the data $(x,1/2, y+1)$, and consider the $\ph_2$ expansion of $F(x)$. Let $N_2$ be the principal $\ph_2$-Newton polygon of $F(x)$ with respect to $\om$.
	Since $2$ is the  degree of $S_2$, then $S_2$ can provide  a unique prime ideal of $\Z_K$ lying above $2$ with residue degree $2$ and ramification index $2$ or two prime ideals of $\Z_K$ lying above $2$ with residue degree $1$ and ramification index $2$ each. It follows that  $2$ divides $i(K)$ if and only if $S_2$ provides two prime ideals of $\Z_K$ lying above $2$ with residue degree $1$ and ramification index $2$ each.  We have the following sub-cases:
		\begin{enumerate}
		\item
	    If $\nu_2(b)\ge 4$ and $a\equiv 4\md{32}$, then for $\ph_2=x^2+2x+2$, we have 	$F(x)=(x-4)\ph_2^2+8(1+x)\ph_2+(a-4)x+b$ is the $\ph_2$ expansion of $F(x)$. Since $\om((x-4)\ph_2^2)=5$ and $\om(8(1+x)\ph_2)=8$, we conclude that  if $\nu_2(b)\in\{4,5\}$, then $N_2$ has a single side of degree $1$. If $\nu_2(b)\ge 6$ and  $\nu_2(a-4)=5$, then  $N_2$ has a single side of degree $2$ and $S_2$ provides a unique ideal of $\Z_K$ lying above $2$. If   $\nu_2(b)\ge 6$ and $\nu_2(a-4)\ge 6$, then  $N_2$ has  two sides, and so $2\Z_K=\p_1\p_2^2\p_3^2$ with residue degree $1$ each. 
	    \item      
	      If  $\nu_2(b)\ge 4$ and $a\equiv 20\md{32}$, then for  $\ph_2=x^2-2x-2$, we have $F(x)=(x+4)\ph_2^2+(16x+24)\ph_2+(a+44)x+b+32$ is the $\ph_2$ expansion of $F(x)$. Since $\om((x+4)\ph_2^2)=5$ and $\om(16x+24)\ph_2)=8$, we conclude that 
        if $\nu_2(b)=4$, then $N_2$ has a single side of degree $1$, and so $2\Z_K=\p_1\p_2^4$ with residue degree $1$ each prime ideal factor. If $\nu_2(b)\ge 5$ and  $\nu_2(a-4)=5$, then  $N_2$ has a single side of degree $2$ and $S_2$ provides a unique ideal of $\Z_K$ lying above $2$. More precisely, $2\Z_K=\p_1\p_2^2$ with $f_2=2$.  If   $\nu_2(b)=5$ and $\nu_2(a+44)\ge 6$, then  $N_2$ has  two sides, and so $2\Z_K=\p_1\p_2^2\p_3^2$ with residue degree $1$ each. 
       	    	    \item
	    If $\nu_2(b)=4$ and $a\equiv 12\md{16}$, then for $\ph_2=x^2+2$, we have 	$F(x)=x\ph_2^2-4x\ph_2+(a+4)x+b$ is the $\ph_2$ expansion of $F(x)$, and so $N_2$ has a single side of degree $1$.
	    \item
	    If $\nu_2(b)\ge 5$ and $a\equiv 12\md{32}$, then for $\ph_2=x^2+2$, we have 	$F(x)=x\ph_2^2-4x\ph_2+(a+4)x+b$ is the $\ph_2$ expansion of $F(x)$ and $N_2$ has a single side of degree $2$ and $R_{\la}(F)(y)=y^2+y+1$, which is irreducible over $\fph$.	    
	    \item
	    If $\nu_2(b)\ge 5$ and $a\equiv 28\md{32}$, then for $\ph_2=x^2+2$, we have 	$F(x)=x\ph_2^2-4x\ph_2+(a+4)x+b$ is the $\ph_2$ expansion of $F(x)$.  Since $\om(x\ph_2^2)=5$,  $\om(-4x\ph_2)=7$  and $\om((a+4)x+b)\ge 10$, we conclude that    $N_2$ has two  sides of degree $1$ each. Thus $2\Z_K=\p_1\p_2^2\p_3^2$ with residue degree $1$ each.
	\end{enumerate}
			\item 
	If	$2$ does not divide $a$ and $2$  divide $b$, then  $ \ol{F(x)}= x\ph^4$ in $\F_2[x]$, where $\ph = x-1$. Since $N_{x}^+(F)$ has a single side of length one, $x$ provides a unique prime ideal of $\Z_K$ lying above $2$ with residue degree $1$.
	
	 For $\ph$, let 
	$F(x)=\ph^5+5\ph^4+10\ph^3+10\ph^2+(5+a)\ph+(1+a+b)$. It follows that:	
		\begin{enumerate}
	    \item
	     If $b\equiv 1-a \md4$, then $\nu_2(1+a+b)=1$, and so  $2$ does not divide $ind(\al)$, and so $2$ does not divide $i(K)$.	      
	    	     	       \item
	     If $b\equiv 3-a \md4$ and $a\equiv 1\md4$, then $\nu_2(1+a+b)\ge  2$ and $\nu_2(5+a)=1$. Thus $\npp{F}=S_1+S_2$ has two sides joining  $(0,v)$, $(1,1)$ and $(4,0)$ with $v\ge 2$. So the sides of $\npp{F}$ are of  degree $1$ each. Hence $\ph$ provides  two prime ideals of $\Z_K$ lying above $2$ with residue degree $1$ each. Therefore $2\Z_K=\p_{111}\p_{211}\p_{221}^3$ with residue degree $1$ each prime ideal factor. It follows that $2$ divides $i(K)$ and $\nu_2(i(K))=1$.	      	      
	      \item
	      If $b\equiv 3-a \md8$ and $a\equiv 3\md4$, then
	      $\nu_2(1+a+b)= 2$ and $\nu_2(5+a)\ge 2$. Thus $\npp{F}=S$ has a single side joining  $(0,2)$, $(2,1)$ and $(4,0)$, with $F_{S}(y)=y^2+y+1$  irreducible over $\fph$. Therefore  $\ph$ provides a unique prime ideal of $\Z_K$ lying above $2$ with residue degree $2$. Hence   $2\Z_K=\p_{111}\p_{211}^2$ with residue degrees $f_{111}=1$ and $f_{211}=2$. Consequently, $2$ does not divide $i(K)$.	      
	      \item
	      If $b\equiv -(1+a) \md{16}$ and $a\equiv 7\md8$, then
	      $\nu_2(1+a+b)\ge 4$ and $\nu_2(5+a)= 2$. Thus $\npp{F}=S_1+S_2+S_3$ has three  sides joining  $(0,v)$, $(1,2)$, $(2,1)$ and $(4,0)$ with $v\ge 4$. Thus every side  is of degree $1$. Therefore   $2\Z_K=\p_{111}\p_{211}^2\p_{221}\p_{231}$ with residue degrees $f_{111}= f_{211}=f_{221}=f_{231}=1$. Consequently, $\nu_2(i(K))=2$.
	      \item
	      If $b\equiv 8-(1+a) \md{16}$ and $a\equiv 7\md8$, then
	      $\nu_2(1+a+b)=3$ and $\nu_2(5+a)= 2$. Thus $\npp{F}=S_1+S_2$ has two  sides joining  $(0,3)$, $(1,2)$, $(2,1)$ and $(4,0)$ such that  $F_{S_1}(y)=y^2+y+1$  irreducible over $\fph$ and $F_{S_1}(y)$ is of degree $1$. Therefore   $2\Z_K=\p_{111}\p_{211}^2\p_{221}$ with residue degrees $f_{111}= f_{211}=1$ and  $f_{221}=2$. Consequently, $2$ does not divide $i(K)$.	      
	       	       \item
	      If $b\equiv 8-(1+a) \md{16}$ and $a\equiv 3\md{8}$, 
	      then $\nu_2(\triangle(F))\ge 11$. As in the proof of Theorem \ref{mono}, let $b_2=\frac{b}{4}$ and $u=\frac{-5b_2}{a}$. Since $2$ does not divide $a$, then $u\in\Z_2$. Let $F(x+u)=x^5+5ux^4+10u^2x^3+10u^3x^3+Ax+B$  with $A=5u^4+a$ and $B=b+au+u^5$. As we did  in the proof of Theorem \ref{mono}, we have $A=\frac{ 2^8a^5+5^5b^4}{(4a)^4}=\frac{ \triangle(F)}{2^{8}a^4}$ and $B=F(u)=\frac{-b_2(2^8a^5+5^5b^4}){2^{8}a^5}= -b_2\frac{ \triangle(F)}{2^{8}a^5}$. Thus $\nu_2(A)=\nu_2(B)=\nu_2(\triangle(F))-8$.
	      	       It follows that if $\nu_2(\triangle(F))$ is even, then  for $\ph=x-u$, $\npp{F}=S_1+S_2$ has two sides joining $(0,\nu_p(\triangle(F))-8)$, $(2,1)$ and $(4,0)$. So,  $2\Z_K=\p_{111}\p_{211}^2\p_{221}^2$ with residue degrees $f_{111}= f_{211}=f_{221}=1$. Consequently, $\nu_2(i(K))=1$.
	       If $\nu_2(\triangle(F))=2k+9$ is odd, then  for $\ph=x-u$, $\npp{F}=S_1+S_2$ has two sides joining $(0,2k+1)$, $(2,1)$ and $(4,0)$. 
        Let us replace $\ph$ by $x-(u+2^k)$ and consider $F(x+u+2^k)=x^5+5(u+2^k)x^4+10(u+2^k)^2x^3+10(u+2^k)^3x^3+Ax+B$  with $A=5(u+2^k)^4+a$ and $B= b+a(u+2^k)+(u+2^k)^5$. Thus $A\equiv 5u^4+a+2^{k+2}u^3 \md{2^{2k+1}}$ and  $B\equiv F(u)+2^k\cdot (a+5u^4)+5\cdot 2^{2k+1}u^3 \md{2^{4k+1}}$. Therefore, $\nu_2(A)=k+2$ and $\nu_2(B)\ge 2k+2$.  
	       Since $F(u)=-b_2\frac{ \triangle(F)}{2^{8}a^5}$ and $u^3\equiv -b_2a\md{4}$, we conclude that  $B\equiv 2^{2k+1}(-b_2\triangle_2+5u^3)\md{2^{3k+1}}$, where $\triangle_2=\frac{\triangle(F)}{2^{2k+9}}$. Hence  if $\triangle_2\equiv a \md4$, then $\nu_2(B)= 2k+2$, and so $\npp{F}$ has two sides joining $(0,2k+2)$, $(1,k+2)$, and $(4,0)$ with $F_{S_1}(y)=y^2+y+1$, which is irreducible over $\fph$. Thus $2\Z_K=\p_{111}\p_{211}^2\p_{221}$  with residue degrees $f_{111}= f_{211}=1$ and $f_{221}=2$. In this case  $2$ does not divide $i(K)$.	      	       
	       If $\triangle_2\equiv -a \md4$, then $\nu_2(B)\ge 2k+3$, and so $\npp{F}$ has three sides joining $(0,v)$, $(1,k+2)$, and $(4,0)$ with $v\ge 2k+3$. Thus $2\Z_K=\p_{111}\p_{211}^2\p_{221}\p_{231}$ with residue degree $1$ each prime ideal factor. Consequently, $\nu_2(i(K))=2$.
	       	       \item 
	      If $b\equiv 16-(1+a) \md{32}$ and $a\equiv 3\md{16}$, then
	      $\nu_2(1+a+b)= 4$ and $\nu_2(5+a)= 3$. Thus $\npp{F}=S_1+S_2$ has two  sides joining  $(0,4)$, $(2,1)$ and $(4,0)$. Thus every side  is of degree $1$. Therefore   $2\Z_K=\p_{111}\p_{211}^2\p_{221}^2$ with residue degrees $f_{111}= f_{211}=f_{221}=1$. Consequently, $\nu_2(i(K))=1$.
	       \item
	      If $b\equiv 32-(1+a) \md{64}$ and $a\equiv 3\md{16}$, then
	      $\nu_2(1+a+b)= 5$ and $\nu_2(5+a)= 3$. Thus $\npp{F}=S_1+S_2$ has two  sides joining  $(0,5)$, $(1,3)$, $(2,1)$ and $(4,0)$. Thus  $F_{S_1}(y)=y^2+y+1$  irreducible over $\fph$ and $S_2$ is  of degree $1$. Therefore   $2\Z_K=\p_{111}\p_{211}^2\p_{221}$ with residue degrees $f_{111}= f_{211}=1$ and  $f_{221}=2$. Consequently, $2$ does not divide $i(K)$.	      
	           \item
	      If $b\equiv -(1+a) \md{64}$ and $a\equiv 3\md{16}$, then
	      $\nu_2(1+a+b)\ge 6$ and $\nu_2(5+a)= 3$. Thus $\npp{F}=S_1+S_2+S_3$ has three  sides joining  $(0,v)$, $(1,3)$, $(2,1)$ and $(4,0)$ with $v\ge 6$. Thus every side  is of degree $1$. Therefore   $2\Z_K=\p_{111}\p_{211}^2\p_{221}\p_{231}$ with residue degrees $f_{111}= f_{211}=f_{221}=f_{231}=1$. Consequently, $\nu_2(i(K))=2$.
	     	       \item
	      If $b\equiv 16-(1+a) \md{32}$ and $a\equiv 11\md{16}$, then
	      $\nu_2(1+a+b)= 4$ and $\nu_2(5+a)\ge 4$. Thus $\npp{F}=S_1+S_2$ has two  sides joining  $(0,4)$, $(2,1)$ and $(4,0)$. Thus every side  is of degree $1$. Therefore   $2\Z_K=\p_{111}\p_{211}^2\p_{221}^2$ with residue degrees $f_{111}= f_{211}=f_{221}=1$. Consequently, $\nu_2(i(K))=1$.
	       \item
	      If $b\equiv -(1+a) \md{32}$ and $a\equiv 11\md{16}$, then $\nu_2(\triangle(F))\ge 12$. Under the notation of the item $(2) (f)$, if $\nu_2(\triangle(F))$ is even, then  for $\ph=x-u$, $\npp{F}=S_1+S_2$ has two sides joining $(0,\nu_p(\triangle(F))-8$, $(2,1)$ and $(4,0)$, and so  $2\Z_K=\p_{111}\p_{211}^2\p_{221}^2$ with residue degrees $f_{111}= f_{211}=f_{221}=1$. Consequently, $\nu_2(i(K))=1$.
	       If $\nu_2(\triangle(F))=2k+9$ is odd and $\triangle_2\equiv a \md4$, then  $\nu_2(i(K))=0$.	       
	       If $\nu_2(\triangle(F))=2k+9$ is odd and $\triangle_2\equiv -a \md4$, then $\nu_2(i(K))=2$.
	     	\end{enumerate}		
\end{enumerate}
\end{proof}
   {\bf Conflict of interest}\\
  Not Applicable.\\
  {\bf Data availability}\\
  Not applicable.

\end{document}